\documentclass[11pt]{article}
\usepackage[a4paper, %
left=2.5cm, %
right=2.5cm, %
top=2cm, %
bottom=2cm, %
heightrounded]{geometry}
\usepackage[T1]{fontenc}
\usepackage[english]{babel}
\usepackage{amsmath}
\usepackage{amsfonts}
\usepackage{amssymb}
\usepackage{amsthm}
\usepackage{titlesec}
\usepackage{indentfirst}
\usepackage{color}
\usepackage{enumerate}
\usepackage{hyperref}

\newtheorem{Theo}{Theorem}[section]
\newtheorem{Lem}[Theo]{Lemma}
\newtheorem{Prop}[Theo]{Proposition}
\newtheorem{Cor}[Theo]{Corollary}

\theoremstyle{definition}
\newtheorem{Def}[Theo]{Definition}
\newtheorem{Rem}[Theo]{Remark}
\newtheorem{Ex}[Theo]{Example}

\numberwithin{equation}{section}

\date{}
\title{}

\begin{document}

\vspace*{3cm}

\begin{center}
\baselineskip=25pt {\Large \textbf{JAMISON SEQUENCES IN COUNTABLY INFINITE DISCRETE ABELIAN GROUPS}}
\end{center}

\vspace*{1cm}

\begin{center}
\large \textit{Vincent Devinck}
\end{center}

\vspace*{1cm}

\begin{abstract}
\noindent We extend the definition of Jamison sequences in the context of topological abelian groups. Then we study such sequences when the abelian group is discrete and countably infinite. An arithmetical characterization of such sequences is obtained, extending the result of Badea and Grivaux \cite{BadeaGrivaux2} about Jamison sequences of integers. In particular, we prove that the sequence consisting of all elements of the group is a Jamison sequence. In the opposite, a sequence which generates a subgroup of infinite index in the group is never a Jamison sequence.
\end{abstract}

\vspace*{1cm}

\noindent Mathematics Subject Classification (2010) codes: 47A10, 37C85, 43A40

\vspace*{1cm}

\section{Introduction}

To begin, let us recall the original definition of Jamison sequences \cite{BadeaGrivaux1} in the context of bounded linear operators (and $C_0$-semigroups) and some important results around these sequences.

\subsection{Integer Jamison sequences, Jamison sequences for semigroups}

The subject of Jamison sequences has been first studied for sequences of integers by Jamison \cite{Jamison}, Ransford \cite{Ransford}, Ransford and Roginskaya \cite{RansfordRoginskaya} and Badea and Grivaux \cite{BadeaGrivaux1} where the authors studied the relationship between the growth of the sequence $(\Vert T^n\Vert)_{n\geqslant 0}$ of the norms of the iterates of a bounded linear operator $T$ acting on a separable Banach space $X$ and the size of the \textit{unimodular point spectrum} $\sigma_p(T)\cap \mathbb{T}$ of $T$ where $\mathbb{T}$ is the closed unit circle of the complex plane and
$$
\sigma_p(T)\cap \mathbb{T}:=\big\{\lambda \in \mathbb{T}\,\big|\,\mathrm{Ker}(T-\lambda\, \mathrm{Id}_X)\ne\{0\}\big\}
$$ 
is the set of eigenvalues of modulus $1$ (or \textit{unimodular eigenvalues}) of $T$. 

The first important result in this way is due to Jamison \cite{Jamison} in 1965. He proved that the unimodular point spectrum $\sigma_p(T)\cap \mathbb{T}$ of $T$ is at most countable when $T$ is a bounded linear operator acting on a separable space $X$ which is \textit{power-bounded}, that is $\sup_{n\geqslant 0}\Vert T^n\Vert<+\infty$. In the same spirit of Jamison's result, Nikolskii \cite{Nikolskii} proved that if $T$ acts on a separable Hilbert space and $\sigma_p(T)\cap \mathbb{T}$ has positive Lebesgue measure, then the series $\sum_{n\geqslant 0}^{}\Vert T^n\Vert^{-2}$ is convergent.

It was then a natural problem to investigate the influence of \textit{partial power-boundedness} on the size of the unimodular point spectrum of an operator. Ransford \cite{Ransford} and Ransford and Roginskaya \cite{RansfordRoginskaya} proved in particular that partial power-boundedness of an operator does not necessarily imply countability of the unimodular point spectrum. In this way, Badea and Grivaux give in \cite{BadeaGrivaux1} the following definition. 

\begin{Def}[\cite{BadeaGrivaux1}, Definition 1.2]
A sequence of positive integers $(n_k)_{k\geqslant 0}$ is a Jamison sequence if for every infinite-dimensional separable complex Banach space $X$ and every bounded linear operator $T$ on $X$ which is partially power-bounded with respect to the sequence $(n_k)_{k\geqslant 0}$, that is $\sup_{k\geqslant 0}\Vert T^{n_k}\Vert<+\infty$, the unimodular point spectrum $\sigma_p(T)\cap \mathbb{T}$ of $T$ is at most countable.
\end{Def}

It was first notice that the fact that a sequence $(n_k)_{k\geqslant 0}$ is a Jamison sequence or not depends in particular on the growth of the sequence. For instance it is proved in \cite{BadeaGrivaux1} that $(n_k)_{k\geqslant 0}$ is a Jamison sequence as soon as $\sup_{k\geqslant 0}\frac{n_{k+1}}{n_k}<+\infty$ whereas Ransford and Roginskaya proved that the sequence $(2^{2^k})_{k\geqslant 0}$ fails to be a Jamison sequence. For more results around the influence of the growth of the sequences on Jamison sequences, see \cite{Ransford}, \cite{RansfordRoginskaya} and \cite{BadeaGrivaux1}.

In \cite{BadeaGrivaux2}, Badea and Grivaux give an arithmetical characterization of Jamison sequences. Under the (non-restrictive) condition $n_0=1$, one can define a distance $d_{(n_k)}$ on $\mathbb{T}$ by setting
$$
\forall (\lambda,\mu)\in \mathbb{T}^2,\qquad d_{(n_k)}(\lambda,\mu)=\sup_{k\geqslant 0}|\lambda^{n_k}-\mu^{n_k}|
$$
Note that $d_{(n_k)}(\lambda,\mu)$ is the uniform norm along the sequence $(n_k)_{k\geqslant 0}$ of $\chi_\lambda-\chi_\mu$ where $\chi_\lambda : n \longmapsto \lambda^n$ and $\chi_\mu : n \longmapsto \mu^n$ are two (continuous) group homomorphisms of $\mathbb{Z}$. The characterisation of Jamison sequences runs as follows.

\begin{Theo}[\cite{BadeaGrivaux2}, Theorem 2.1]\label{Zjamison}
Let $(n_k)_{k\geqslant 0}$ be an increasing sequence of positive integers with $n_0=1$. The following assertions are equivalent$:$
\begin{enumerate}
\item[$(1)$] $(n_k)_{k\geqslant 0}$ is a Jamison sequence$;$
\item[$(2)$] there exists a positive real number $\varepsilon$ such that for every $\lambda\in \mathbb{T}\setminus\{1\}$,
\begin{equation}\label{Zcharacterization}
\sup_{k\geqslant 0}\vert \lambda^{n_k}-1\vert\geqslant \varepsilon
\end{equation}
\end{enumerate}
\end{Theo}

The hard part of the proof of Theorem \ref{Zjamison} is the following: when condition $(2)$ is not satisfied, Badea and Grivaux had to construct an explicit separable Banach space $X$ and a bounded linear operator $T$ on $X$ which is partially power-bounded with respect to the sequence $(n_k)_{k\geqslant 0}$ and such that the set $\sigma_p(T)\cap \mathbb{T}$ is uncountable. This method of construction will be used in Theorem \ref{characterization}.

The notion of Jamison sequences has also been studied in \cite{Devinck1} in the context of $C_0$-semigroups of bounded linear operators. An increasing sequence of positive real numbers $(t_k)_{k\geqslant 0}$ (which tends to infinity) is called a Jamison sequence for $C_0$-semigroups if for every infinite-dimensional separable complex Banach space $X$ and every $C_0$-semigroup $(T_t)_{t\geqslant 0}$ of bounded linear operators on $X$ (with infinitesimal generator $A$), the set $\sigma_p(A)\cap i\mathbb{R}$ is at most countable. We have the following characterization of Jamison sequences in this context. We denote by $\Vert x \Vert$ the distance of the real number $x$ to the integers, that is 
$$
\Vert x \Vert=\inf\big\{|x-n|\,\big|\,n\in \mathbb{Z}\big\}
$$

\begin{Theo}[\cite{Devinck1}, Theorem 3.3]\label{Rjamison}
Let $(t_k)_{k\geqslant 0}$ be an increasing sequence of positive real numbers such that $t_0=1$ and $\underset{k\to +\infty}{\lim}t_k=+\infty$. The following assertions are equivalent$:$
\begin{enumerate}
\item[$(1)$] $(t_k)_{k\geqslant 0}$ is a Jamison sequence for $C_0$-semigroup$;$
\item[$(2)$] there exists $\varepsilon>0$ such that for every $\theta\in \big]0,\frac{1}{2}\big]$,
$$
\sup_{k\geqslant 0}\Vert t_k\theta\Vert\geqslant \varepsilon
$$
\end{enumerate}
\end{Theo}

It is easy to check that condition $(2)$ is equivalent to the following:
\begin{enumerate}
\item[$(3)$] there exists $\varepsilon>0$ such that for every $x\in \big]0,\frac{1}{2}\big]$, 
\begin{equation}\label{Rcharacterization}
\sup_{k\geqslant 0}\big\vert e^{2i\pi xt_k}-1\big\vert\geqslant \varepsilon
\end{equation}
\end{enumerate}
Condition $(3)$ is a characterization of Jamison sequences of $C_0$ semigroups of bounded linear operators in terms of continuous group homomorphisms $\chi_x : y \longmapsto e^{2\pi i xy}$ of the group $\mathbb{R}$.

We define in the next section the notion of Jamison sequences $(g_k)_{k\geqslant 0}$ belonging to some general topological abelian group $G$.

\subsection{Jamison sequences in topological abelian groups}

In order to define the analog of Jamison sequences in topological abelian groups, we first have to introduce the object which will play the role of bounded linear operator in the context of integer Jamison sequences. If $X$ is a Banach space, we denote by $\mathcal{GL}(X)$ the group of invertible bounded linear operators on $X$.

\begin{Def}
Let $G$ be a topological abelian group. A \textit{representation of} $G$ is a couple $(X,\rho)$ where $X$ is a separable infinite-dimensional complex Banach space and $\rho : G \longrightarrow \mathcal{GL}(X)$ is a continuous homomorphism of groups.
\end{Def}

If $T$ is an invertible bounded linear operator acting on some Banach space $X$, then a complex number $\lambda$ is an eigenvalue of $T$ if and only if there exists a vector $e_\lambda\in X\setminus\{0\}$ such that $T^ne_\lambda=\lambda^n e_\lambda$ for every integer $n$. We then have the following definition in the context of representation of group. If $G$ is a topological abelian group then the Pontryagin dual (or topological dual) of $G$ is the abelian group $\widehat{G}$ consisting in all continuous homomorphisms $\chi$ from $G$ to the unit circle $\mathbb{T}$ of the complex plane. An element of $\widehat{G}$ will be called a \textit{character} of $G$.

\begin{Def}
Let $G$ be a topological abelian group and $(X,\rho)$ be a representation of $G$. A character $\chi\in \widehat{G}$ is called a \textit{unimodular eigenvalue of} $\rho$ if there exists a vector $e_\chi\in X\setminus\{0\}$ such that
$$
\forall g\in G,\qquad \rho(g)e_\chi=\chi(g)e_\chi
$$
The set of unimodular eigenvalues of $\rho$ is denoted by $\sigma_p(\rho)\cap \widehat{G}$ and is called \textit{the unimodular point spectrum} of $\rho$.
\end{Def}

We are also interested in the relationship between \textit{partial boundedness} of a representation $(X,\rho)$ of a group $G$ with respect to a sequence of $G$ and the size of the unimodular point spectrum of $\rho$.

\begin{Def}
Let $G$ be a topological abelian group, $(g_k)_{k \geqslant 0}$ be a sequence of elements of $G$ and $(X,\rho)$ a representation of $G$. We say that $\rho$ is \textit{partially bounded with respect to the sequence} $(g_k)_{k\geqslant 0}$ if $\sup_{k\geqslant 0}\Vert \rho(g_k)\Vert<+\infty$.
\end{Def}

Sequences of elements of $G$ for which partial boundedness implies countability of the unimodular point spectrum will be called $G$-Jamison sequences. More precisely, we have the following definition.

\begin{Def}
Let $G$ be a topological abelian group and $(g_k)_{k\geqslant 0}$ be a sequence of elements of $G$. We say that $(g_k)_{k\geqslant 0}$ is a $G$-Jamison sequence if for every representation $(X,\rho)$ of the group $G$ which is partially bounded with respect to the sequence $(g_k)_{k\geqslant 0}$, the unimodular point spectrum $\sigma_p(\rho)\cap \widehat{G}$ of $\rho$ is at most countable.
\end{Def}

Even if Theorems \ref{Zjamison} and \ref{Rjamison} seem to be only concerned in sequences of the semigroups $\mathbb{N}$ and $\mathbb{R}_+$ respectively, one can easily verify that the proofs of this theorems are also valid in $\mathbb{Z}$ and $\mathbb{R}$ and we have the same characterizations \eqref{Zcharacterization} and \eqref{Rcharacterization} of $\mathbb{Z}$-Jamison sequences and $\mathbb{R}$-Jamison sequences. Let us mention some other examples of $G$-Jamison sequences (or not).

\begin{Ex}
A characterization of $G$-Jamison sequences has been obtained in \cite{Devinck2} for sequences which belong to a finitely generated abelian group $G$:

\begin{Theo}[\cite{Devinck2}, Theorem 10.4.1]
Let $G$ be the finitely generated abelian group 
$$
\mathbb{Z}^\ell\times \mathbb{Z}/a_1\mathbb{Z}\times\dots\times \mathbb{Z}/a_r\mathbb{Z}
$$
where $a_1,\dots,a_r$ are integers greater or equal than $2$. Let $(g_k)_{k\geqslant 1}$ be a sequence of elements of $G$ such that 
$$
\forall k\in \{1,\dots,\ell+r\},\qquad g_k=(\underbrace{0,\dots,0,1}_{k\textrm{ times}},0,\dots,0)
$$
Then the following assertions are equivalent$:$
\begin{enumerate}
\item[$(1)$] $(g_k)_{k\geqslant 1}$ is a $G$-Jamison sequence$;$
\item[$(2)$] there exists $\varepsilon>0$ such that for every $\chi\in \widehat{G}\setminus\{\mathbf{1}\}$, we have
$$
\sup_{k\geqslant 1}|\chi(g_k)-1|\geqslant \varepsilon
$$
\end{enumerate}
where $\widehat{G}$ denote the group of all homomorphisms from $G$ to $\mathbb{T}$.
\end{Theo}
\end{Ex}

\begin{Ex}
If $G$ is any finite abelian group, then every sequence of elements of $G$ is a $G$-Jamison sequence (since $\widehat{G}\simeq G$ is a finite group).
\end{Ex}

\begin{Ex}
If $G=(\mathbb{Z}/2\mathbb{Z})^\mathbb{N}$ is the group of infinite sequences of zeroes and ones, then $\widehat{G}\simeq \bigoplus_{n\geqslant 0}\mathbb{Z}/2\mathbb{Z}$ consists in all infinite sequences with finite number of ones. In particular $\widehat{G}$ is countable. Hence every sequence of elements of $G$ is a $G$-Jamison sequence. More generally, every sequence of a compact group $G$ with countable Pontryagin dual is a $G$-Jamison sequence.
\end{Ex}

Let us now consider the more interesting example of the discrete group $G=\bigoplus_{n\geqslant 0}\mathbb{Z}:=\mathbb{Z}^\infty$ which consists of all infinite sequences of integers with finitely non-zero components. In such a group, it is not difficult to give sequences which are not $\mathbb{Z}^\infty$-Jamison sequences. 

\begin{Ex}\label{exempleZinfini}
Let $G=\mathbb{Z}^\infty$. The Pontryagin dual of $G$ is $\widehat{G}=\mathbb{T}^\mathbb{N}$.
\begin{enumerate}
\item If $(e_k)_{k\geqslant 0}$ denote the canonical basis of $G$ (where $e_k$ is the $k^\textrm{th}$ Kronecker symbol), then $(e_k)_{k\geqslant 0}$ is not a $G$-Jamison sequence.
\item If the sequence $(g_k)_{k\geqslant 0}$ \textit{does not visit} at least one component of $\mathbb{Z}^\infty$ then $(g_k)_{k \geqslant 0}$ is not a $G$-Jamison sequence. For instance, if $(g_k)_{k\geqslant 0}$ is any sequence such that $\{g_k\,|\,k\geqslant 0\}=\bigoplus_{n\geqslant 1}\mathbb{Z}e_k$ (that is the subgroup of $G$ generated by the sequence $(e_k)_{k\geqslant 1}$) then $(g_k)_{k\geqslant 0}$ is not a $G$-Jamison sequence.
\end{enumerate}
\end{Ex}

\begin{proof}
A representation $(X,\rho)$ of $G$ is completely determined by the operators $\rho(e_k)\in \mathcal{GL}(X)$ $(k\geqslant 0)$. Take any separable complex Banach space $X$ and any invertible operator $T$ on $X$ such that $\sigma_p(T)\cap \mathbb{T}$ is uncountable (for instance $X=\ell^2(\mathbb{Z})$ and $T$ is the backward shift defined by $T(x_n)_{n\in \mathbb{Z}}=(x_{n+1})_{n\in \mathbb{Z}}$). Then the representation $(X,\rho)$ of $G$ which is defined by $\rho(e_0)=T$ and $\rho(e_k)=\mathrm{Id}_X$ when $k$ is a positive integer is clearly partially bounded with respect to the sequence $(e_k)_{k\geqslant 0}$ (and with respect to the sequence $(g_k)_{k\geqslant 0}$) and the unimodular point spectrum of $\rho$ is
$$
\sigma_p(\rho)\cap\mathbb{T}^\mathbb{N}=\big\{(m_k)_{k\geqslant 0} \longmapsto \lambda^{m_k}\,\big|\,\lambda\in \sigma_p(T)\cap \mathbb{T}\big\}
$$
Since $\sigma_p(T)\cap \mathbb{T}$ is uncountable, the set $\sigma_p(\rho)\cap \mathbb{T}^\mathbb{N}$ is uncountable.
\end{proof}

\subsection{Organization of the paper}

Our goal is to characterize $G$-Jamison sequences when $G$ is a countably infinite discrete abelian group. As in the case of the group $\mathbb{Z}$ of integers, we show that being a $G$-Jamison sequence or not depends on the arithmetical properties of the sequence $(g_k)_{k\geqslant 0}$. The main result of the paper is the following.

\begin{Theo}\label{theorem}
Let $G$ be a countably infinite discrete abelian group. Let $(g_k)_{k\geqslant 0}$ be a sequence of elements of $G$.
\begin{enumerate}
\item[$(1)$] When the subgroup $G_0$ of $G$ generated by the sequence $(g_k)_{k\geqslant 0}$ is of infinite index in $G$ $($that is if the factor group $G/G_0$ is infinite$)$ then $(g_k)_{k\geqslant 0}$ is not a $G$-Jamison sequence.
\item[$(2)$] If the index $[G:G_0]$ is finite, then one can assume without loss of generality that $G=G_0$ and $(g_k)_{k\geqslant 0}$ is a $G$-Jamison sequence if and only if there exists a positive real number $\varepsilon$ such that for every character $\chi\in \widehat{G}\setminus\{\mathbf{1}\}$,
\begin{equation}\label{arithmetical}
\sup_{k\geqslant 0}|\chi(g_k)-1|\geqslant \varepsilon
\end{equation}
\end{enumerate}
\end{Theo}

In section 2, we investigate condition \eqref{arithmetical}. We show in Proposition \ref{conditionsuffisante} that $(g_k)_{k\geqslant 0}$ is a $G$-Jamison sequence as soon as the subgroup $G_0$ generated by $(g_k)_{k\geqslant 0}$ is of finite index in $G$ and when the characters of $G_0$ are uniformly separated by the sequence $(g_k)_{k\geqslant 0}$. A consequence of this is that the whole sequence of elements of the countable group $G$ is a Jamison sequence (Corollary \ref{exemplejamison}), extending the seminal result of Jamison.

The first main difference with the group of integers is that , given a sequence $(g_k)_{k\geqslant 0}$ of elements of a countably infinite discrete abelian group $G$, the map
\begin{align}\label{distance}
d_{(g_k)} : \left\{\begin{array}{c c c}
\widehat{G}\times \widehat{G} & \longrightarrow & \mathbb{R}\\
(\chi,\varphi) & \longmapsto & \sup_{k\geqslant 0}|\chi(g_k)-\varphi(g_k)|
\end{array}\right.
\end{align}
cannot be made into a distance on $\widehat{G}$ since $G$ is not finitely generated in general. In section 3, we show that when $d_{(g_k)}$ is \textit{very far} from being a distance on $\widehat{G}$, that is if the index $[G:G_0]$ of the subgroup $G_0$ of $G$ generated by the sequence $(g_k)_{k\geqslant 0}$ is infinite, then $(g_k)_{k\geqslant 0}$ is not a $G$-Jamison (Theorem \ref{nonjamison}). This phenomenon generalizes what we observe in example \ref{exempleZinfini}.

In section 4, we then characterize $G$-Jamison sequences $(g_k)_{k\geqslant 0}$ when the subgroup $G_0$ (generated by the sequence $(g_k)_{k\geqslant 0}$) is of finite index in $G$. In this case, one can assume the $(g_k)_{k\geqslant 0}$ is a generate sequence of $G$, that is $G=G_0$ (Proposition \ref{wlog}). We prove in Theorem \ref{characterization} that $(g_k)_{k\geqslant 0}$ is a $G$-Jamison sequence if and only if any two distinct characters of $G$ are (uniformly) $\varepsilon$-separated for the distance $d_{(g_k)}$ for some $\varepsilon>0$.

In the rest of the paper, $G$ is a countably infinite discrete abelian group. Note that, in this context, every group homomorphism $\chi$ from $G$ to the unit circle $\mathbb{T}$ is automatically continuous. Hence the Pontryagin dual $\widehat{G}$ of $G$ is the group of all (continuous) homomorphisms $\chi$ from $G$ to $\mathbb{T}$. In $\widehat{G}$, the trivial character (identically equal to one) is denoted by $\mathbf{1}$.

If $(g_k)_{k\geqslant 0}$ is a sequence of elements of $G$, we denote by $G_0$ the subgroup of $G$ generated by this sequence, that is the subset of $G$ consisting in all (finite) linear combinations of elements of $(g_k)_{k\geqslant 0}$ with integer coefficients. The index of $G_0$ in $G$ is denoted $[G:G_0]$; recall that $[G:G_0]$ stands for the cardinality of the quotient group $G/G_0$. If $\chi$ and $\varphi$ are two characters of $G$, we put 
$$
d_{(g_k)}(\chi,\varphi):=\sup_{k\geqslant 0}|\chi(g_k)-\varphi(g_k)|
$$

In the next section we investigate the sufficient condition \eqref{arithmetical} of Theorem \ref{theorem}.

\section{An arithmetical sufficient condition}

Let $G$ be a countably infinite discrete abelian group and $(g_k)_{k\geqslant 0}$. We first investigate a sufficient condition which provides $G$-Jamison sequences.

\subsection{A sufficient condition}

We begin with the following general easy lemma.

\begin{Lem}\label{equipotent}
Let $G$ be a topological abelian group, $H$ a subgroup of $G$ and let $\chi$ be a character of $G$. The set 
$$
\Gamma_\chi=\big\{\varphi\in \widehat{G}\,\big|\,\varphi_{\restriction_H}=\chi_{\restriction_H}\big\}
$$
is equipotent to $\widehat{G/H}$.
\end{Lem}

\begin{proof}
It is straightforward to check that 
$$
\vartheta : \left\{\begin{array}{c c c}
\Gamma_\chi & \longrightarrow & \widehat{G/H}\\
\varphi & \longmapsto & (g+H \mapsto \varphi(g)\chi(g)^{-1})
\end{array}\right.
$$
is a well-defined bijective map.
\end{proof}

The sufficient arithmetical condition runs as follows.

\begin{Prop}\label{conditionsuffisante}
Let $G$ be a countably infinite discrete abelian group and $(g_k)_{k\geqslant 0}$ be a sequence of elements of $G$. We denote by $G_0$ the subgroup of $G$ generated by the sequence $(g_k)_{k\geqslant 0}$. We assume that 
\begin{enumerate}
\item[$(1)$] the dual group $\widehat{G/G_0}$ is at most countable$;$
\item[$(2)$] there exists $\varepsilon>0$ such that for every $\chi\in \widehat{G}_0\setminus\{\mathbf{1}\}$, we have $d_{(g_k)}(\chi,\mathbf{1})\geqslant \varepsilon$.
\end{enumerate}
Then $(g_k)_{k\geqslant 0}$ is a $G$-Jamison sequence.
\end{Prop}

\begin{proof}
Let $(X,\rho)$ be a representation of the group $G$ which is partially bounded with respect to the sequence $(g_k)_{k\geqslant 0}$ and let $M:=\displaystyle \sup_{k\geqslant 0}\Vert \rho(g_k)\Vert<+\infty$. Let $\chi$ and $\varphi$ be two disctinct unimodular eigenvalues of $\rho$. We denote by $e_\chi$ et $e_\varphi$ two normalized eigenvectors associated to these eigenvalues:
$$
\Vert e_\chi\Vert=\Vert e_\varphi\Vert=1\qquad \textrm{and for every }g\in G,\qquad 
\left\{\begin{array}{c}
\rho(g)e_\chi=\chi(g)e_\chi\\
\rho(g)e_\varphi=\varphi(g)e_\varphi
\end{array}\right.
$$ 
For every non-negative integer $k$, 
$$
\vert \chi(g_k)-\varphi(g_k)\vert-\Vert e_\chi-e_\varphi\Vert\leqslant \Vert \rho(g_k)(e_\chi-e_\varphi)\Vert\leqslant M\Vert e_\chi-e_\varphi\Vert
$$
Hence, 
$$
\Vert e_\chi-e_\varphi\Vert\geqslant \frac{d_{(g_k)}(\chi,\varphi)}{M+1}\geqslant \frac{\varepsilon}{M+1}
$$
where the second inequality follows from condition $(2)$. Since the space $X$ is separable, the set $$
\big\{\chi_{\restriction_{G_0}}\,\big|\,\chi \in \widehat{G}\big\}
$$ 
is at most countable. We conclude the proof by using assertion $(1)$ and Lemma \ref{equipotent}.
\end{proof}

Note that condition $(1)$ of Proposition \ref{conditionsuffisante} is not a problem in the original context of Jamison sequences.

\begin{Rem}
Our condition on the dual of $G/G_0$ is automatically satisfied for the groups $\mathbb{R}$ and $\mathbb{Z}$.
\begin{enumerate}
\item If $G=\mathbb{Z}$ then there exists a positive integer $n$ such that $G_0=n\mathbb{Z}$. Hence $\widehat{G/G_0}$ is isomorphic to the group of $n^\textrm{th}$ roots of unity.
\item If $G=\mathbb{R}$ then there exists $\theta\in \mathbb{R}^*$ such that $\theta\mathbb{Z}\subset G_0$. We can realized $\mathbb{R}/G_0$ as a quotient of $\mathbb{R}/\theta\mathbb{Z}$ by the third isomorphism theorem:
$$
\mathbb{R}/G_0 \simeq (\mathbb{R}/\theta\mathbb{Z})/(G_0/\theta\mathbb{Z})
$$
We know that $\widehat{\mathbb{R}/\theta\mathbb{Z}}\simeq\mathbb{Z}$ is countable. It then suffices to prove that for any subgroup $H$ of an abelian topological group $G$, the character group of $G/H$ is countable as soon as $\widehat{G}$ is countable. This directly follows from the fact that
$$
\left\{\begin{array}{c c c}
\widehat{G/H} & \longrightarrow & \widehat{G}\\
\chi & \longmapsto & (g \mapsto \chi(g+H))
\end{array}\right.
$$
is an injective map.
\end{enumerate}
\end{Rem}

Proposition \ref{conditionsuffisante} allows us to give a natural example of $G$-Jamison sequence when $G$ is a countably infinite discrete abelian group.

\subsection{A Jamison's type result}

Jamison gives in \cite{Jamison} the first example of $\mathbb{Z}$-Jamison sequence: the sequence of integers is a Jamison sequence in the context of sequences of integers. The same phenomenon occurs in general countably infinite discrete abelian groups.

\begin{Cor}\label{exemplejamison}
Let $G=\{g_k\,|\,k\geqslant 0\}$ be a countably infinite discrete abelian group. Then $(g_k)_{k\geqslant 0}$ is a $G$-Jamison sequence.
\end{Cor}

\begin{proof}
Since $(g_k)_{k\geqslant 0}$ is the sequence of all elements of $G$, note that $d_{(g_k)}$ is a well-defined distance on $\widehat{G}$. According to Proposition \ref{conditionsuffisante}, it suffices to prove that any two distinct characters of $\widehat{G}$ are uniformly separated for the distance $d_{(g_k)}$. We show here that this fact follows from a famous result of Prodanov \cite{Prodanov}. 

For any topological abelian group $H$, let us denote by $H^\ast$ the algebraic dual of $H$ which consists in all group homomorphisms from $H$ to the unit circle $\mathbb{T}$ (whereas $\widehat{H}$ is the group of all such continuous homomorphisms).
\begin{Lem}[Prodanov,\cite{Prodanov}]\label{prodanov}
Let $H$ be a topological abelian group, $U$ an open subset of $H$, let $f$ be a continuous function from $H$ to $\mathbb{C}$ and $M$ a convex closed subset of $\mathbb{C}$. Let $k$ be a positive integer and $\chi_1,\dots,\chi_k$ be elements of $H^\ast$. Assume that there exists $(c_1,\dots, c_k)\in \mathbb{C}^k$ such that $\sum_{j=1}^{k}c_j\chi_j(x)-f(x)$ belongs to $M$ for every $x$ in $U$. If 
$$\{\chi_1,\dots,\chi_n\}=\{\chi_1,\dots,\chi_k\}\cap \widehat{H}$$
with $n\leqslant k$, then $\sum_{1\leqslant j\leqslant n}^{}c_j\chi_j(x)-f(x)$ belongs to $M$ for every $x$ in $U$.
\end{Lem}
Let $\chi$ and $\varphi$ be two distinct characters of our discrete (countably infinite abelian) group $G$ (which are continuous when $G$ is equipped with the discrete topology). We take for $U$ the group $G$ and for $M$ the closed disk with center $0$ and radius 
$$
r=d_{(g_k)}(\chi,\varphi)=\sup_{g\in G}|\chi(g)-\varphi(g)|>0
$$ 
The character $\chi\varphi^{-1}$ is non-constant and $d_{(g_k)}(\chi\varphi^{-1},1)\leqslant r$. We now endow $G$ with the trivial topology (the only open sets of $G$ are $\varnothing$ and $G$). Then the non-constant character $\chi\varphi^{-1}$ is discontinuous for this topology on $G$. According to Lemma \ref{prodanov} (recall that $G$ is assumed to be equipped with the trivial topology), we have the inequality $1\leqslant r$, that is to say $1\leqslant d_{(g_k)}(\chi,\varphi)$. Proposition \ref{conditionsuffisante} implies that $(g_k)_{k\geqslant 0}$ is a $G$-Jamison sequence.
\end{proof}

At this stage of the proof, we investigate the first condition we need in Proposition \ref{conditionsuffisante} on the size of the dual group of $G/G_0$. Indeed example \ref{exempleZinfini} shows that a sequence of $\mathbb{Z}^\infty$ has to generate a \textit{big} subgroup of $\mathbb{Z}^\infty$ to be a $\mathbb{Z}^\infty$-Jamison sequence. In the next section, we prove that this is a general phenomenon in countably infinite discrete abelian groups.

\section{A negative result on Jamison sequences}

The following result states that if the subgroup $G_0$ generated by the sequence $(g_k)_{k\geqslant 0}$ is \textit{small} in $G$, then $(g_k)_{k\geqslant 0}$ is never a $G$-Jamison sequence.

\begin{Theo}\label{nonjamison}
Let $G$ be a countably infinite discrete abelian group and $(g_k)_{k\geqslant 0}$ be a sequence of elements of $G$ such that the subgroup $G_0$ of $G$ generated by this sequence is of infinite index in $G$. Then $(g_k)_{k\geqslant 0}$ is not a $G$-Jamison sequence.
\end{Theo}

\begin{proof}
We need to construct a representation $(X,\rho)$ of the group $G$ which is partially bounded with respect to the sequence $(g_k)_{k\geqslant 0}$ and such that the unimodular point spectrum $\sigma_p(\rho)\cap \widehat{G}$ of $\rho$ is uncountable. 

The beginning of our construction is the same as in the proof of Theorem 2.1 of \cite{BadeaGrivaux2}. Our representation $\rho$ will be a \textit{backward shift} (or \textit{translation representation}) on some suitably chosen $\ell^2$ space over the group $G$. The first difference with the proof of Badea and Grivaux is that we have to built this $\ell^2$ space in such a way that $\rho$ maps every element of $G$ into a bounded linear (invertible) operator on this $\ell^2$ space. To do this, we need to built a positive weighted function $w$ on $G$ which has the property to be \textit{sub-invariant} (see Corollary \ref{subinvariant}).

In the first step of the proof, we built a weighted $\ell^2$ space $\ell^2(G,w)$ on the group $G$ where the weight $w$ is positive and such that we will be able to produce a representation $(\ell^2(G,w),\rho)$ essentially such that $\rho(g)$ is a bounded linear operator on $\ell^2(G,w)$ for every element $g$ of $G$.

\subsubsection*{Step 1 : the weight $w$, the space $\ell^2(G,w)$ and the representation $\rho$.}

The construction of the weight $w$ is inspired from the one used by Glasner and Weiss in \cite{GlasnerWeiss}. Given a countably infinite discrete group $G$ and a representation $(\mathcal{H},\rho)$ of $G$ on a separable Hilbert space $\mathcal{H}$, the authors proved in \cite{GlasnerWeiss} that the representation $(\mathcal{H},\rho)$ is \textit{universal} in the following sense: for every ergodic probability-preserving free action $\tilde{\rho}$ of $G$ on a probability space $(X,\mathcal{B},\mu)$, there exists a Borel probability measure $\nu$ with full support on $(\mathcal{H},\mathrm{Borel}(\mathcal{H}))$ which is $\rho$-invariant and such that the two actions of $\tilde{\rho}$ and $\rho$ on $(X,\mathcal{B},\mu)$ and $(\mathcal{H},\mathrm{Borel}(\mathcal{H}),\nu)$ respectively are isomorphic (see \cite{GlasnerWeiss} for more details).

Since our group $G$ is countable, we can fix an enumaration $G=\{h_n\,\vert\,n\geqslant 1\}$ of this group. For every positive integer $n$, we denote by $\delta_n$ the Dirac measure at point $h_n$, that is
$$
\forall g\in G,\,\qquad \delta_n(g)=\left\{\begin{array}{c l}
1 &\ \textrm{if }g=h_n\\
0 &\ \textrm{otherwise}
\end{array}\right.
$$
We then define a probability measure $m$ on $G$ by putting 
$$
m=\sum_{n=1}^{+\infty}2^{-n}\delta_n
$$
For every positive integer $n$, the notation 
$$m^{\ast n}=\underbrace{m\ast\dots\ast m}_{n\textrm{ times}}$$ 
stands for the $n$-fold interation convolution of the measure $m$. Let us now consider the weight $w$ on $G$ defined by
\begin{equation}\label{weight}
\forall g\in G,\qquad w(g)=\sum_{n=1}^{+\infty}2^{-n}m^{\ast n}(g)
\end{equation}
and the associated $\ell^2$ weighted space 
\begin{equation}\label{l2space}
\ell^2(G,w)=\bigg\{f : G \longrightarrow \mathbb{C}\,\bigg\vert\,\Vert f\Vert^2=\sum_{g\in G}^{}\vert f(g)\vert^2w(g)<+\infty\bigg\}
\end{equation}
Let us first notice that $\widehat{G}$ is a subset of $\ell^2(G,w)$ since $\sum_{g\in G}^{}w(g)<+\infty$ (this sum is equal to $1$). The next lemma shows that the \textit{backward shift representation} is well-defined. 

\begin{Lem}\label{representation}
For every element $g$ of $G$, the map 
$$
\rho(g) : \left\{\begin{array}{c c c}
\ell^2(G,w) & \longrightarrow & \ell^2(G,w)\\
f & \longmapsto & \big(h \mapsto f(h+g)\big)
\end{array}\right.
$$
is a well-defined bounded linear operator on the space $\ell^2(G,w)$.
\end{Lem}
\begin{proof}
Let $k$ be a positive integer. For every element $f$ of $\ell^2(G,w)$, we have 
\begin{align}\label{egalite}
\notag\Vert \rho(h_k)(f)\Vert^2 &=\sum_{g\in G}^{}\vert \rho(h_k)(f)(g)\vert^2 w(g)\\
\notag &=\sum_{g\in G}^{}\vert f(g+h_k)\vert^2 w(g)\\
&=\sum_{n=1}^{+\infty}2^{-n}\sum_{g\in G}^{}\vert f(g+h_k)\vert^2 m^{\ast n}(g)
\end{align}
But for every positive integer $n$, 
\begin{align}\label{inegalite}
\notag\sum_{g\in G}^{}\vert f(g+h_k)\rvert^2 m^{\ast n}(g)&=\sum_{(g_1,\dots, g_n)\in G^n}^{}\vert f(g_1+\dots +g_n+h_k)\rvert^2 m(g_1)\dots m(g_n)\\
\notag &=m(h_k)^{-1}\sum_{(g_1,\dots, g_n)\in G^n}^{}\vert f(g_1+\dots +g_n+h_k)\rvert^2 m(g_1)\dots m(g_n)m(h_k)\\
&\leqslant 2^k\sum_{(g_1,\dots, g_{n+1})\in G^{n+1}}^{}\vert f(g_1+\dots+g_{n+1})\vert^2 m(g_1)\dots m(g_{n+1})
\end{align}
It now follows from \eqref{egalite} and \eqref{inegalite} that
$$
\Vert\rho(h_k)(f)\Vert^2\leqslant 2^{k+1}\sum_{n=1}^{+\infty}2^{-(n+1)}\sum_{g\in G}^{}\vert f(g)\vert^2 m^{\ast (n+1)}(g)\leqslant 2^{k+1}\Vert f\Vert^2
$$
This proves that $\rho(h_k)$ is a bounded linear operator on the space $\ell^2(G,w)$ with $\Vert \rho(h_k)\Vert\leqslant \sqrt{2}^{k+1}$.
\end{proof}

According to Lemma \ref{representation}, the map 
$$
\rho : \left\{\begin{array}{c c c}
G & \longrightarrow & \mathcal{GL}(\ell^2(G,w))\\
g & \longmapsto & \rho(g)
\end{array}\right.\qquad \textrm{where}\qquad \rho(g) : \left\{\begin{array}{c c c}
G & \longrightarrow & \mathbb{C}\\
h & \longmapsto & f(h+g)
\end{array}\right.
$$
is a well-defined representation of $G$. Moreover, every character $\chi$ of $G$ is an eigenvalue of $\rho$ (with eigenvector $\chi$) since for every $(g,h)\in G^2$, 
$$
\rho(g)\chi(h)=\chi(g+h)=\chi(g)\chi(h)
$$
that is $\rho(g)\chi=\chi(g)\chi$. 

The sequel of the proof consists in making the representation $(\ell^2(G,w),\rho)$ partially bounded with respect to the sequence $(g_k)_{k\geqslant 0}$ and with an uncountable unimodular point spectrum.

\subsubsection*{Step 2 : the role of the annihilator group $G_0^\perp$ of $G_0$.}

Recall that $G_0$ is the subgroup of $G$ generated by the sequence $(g_k)_{k\geqslant 0}$. The annihilator of $G_0$ is the subgroup of $\widehat{G}$ defined by 
$$
G_0^\perp=\big\{\chi\in \widehat{G}\,\big|\,\forall g\in G_0,\ \chi(g)=1\big\}
$$
Let us then consider the closed subspace $\mathcal{H}$ of $\ell^2(G,w)$ spanned by this annihilator group, that is
$$
\mathcal{H}:=\overline{\textrm{span}}^{\Vert\cdot \Vert}\big\{\chi\,\big\vert\,\chi\in G_0^\perp\big\}
$$
and the representation $\rho_\mathcal{H}$ induced by $\rho$ on $\mathcal{H}$ that is
$$
\rho_{\mathcal{H}} : \left\{\begin{array}{c c c}
G & \longrightarrow & \mathcal{GL}(\mathcal{H})\\
g & \longmapsto & \rho(g)
\end{array}\right.
$$
In the next lemma, it is shown that $\rho_{\mathcal{H}}$ is a well-defined representation and its partial-boundedness with respect to the sequence $(g_k)_{k\geqslant 0}$ is established.
\begin{Lem}\label{bounded}
The representation $(\mathcal{H},\rho_\mathcal{H})$ is partially bounded with respect to the sequence $(g_k)_{k\geqslant 0}$.
\end{Lem}

\begin{proof}
For every character $\chi$ of $G$, we have the equality $\rho_{\mathcal{H}}(g)\chi=\chi(g)\chi$ for every element $g$ of $G$. In particular, $\mathcal{H}$ is an invariant subspace of $\rho_{\mathcal{H}}(g)$ for every $g$ of $G$.  More precisely,
$$
\forall f\in \mathcal{H},\ \forall g\in G_0,\qquad \rho_{\mathcal{H}}(g)(f)=f
$$
since $\chi(g)=1$ when $(\chi,g)$ belongs to $G_0^\perp\times G_0$ by definition of $G_0^\perp$. In particular, $\rho_{\mathcal{H}}(g_k)=\mathrm{Id}_{\mathcal{H}}$ for every non-negative integer $k$. It is then obvious that $\sup_{k\geqslant 0}\Vert \rho_{\mathcal{H}}(g_k)\Vert=1$ is finite and Lemma \ref{bounded} is proved.
\end{proof}

Recall that $G$ is a discrete group. It is well-known that any factor group of a discrete group is discrete. Hence, the factor group $G/G_0$ is discrete. Furthermore, it follows from the Pontryagin duality (which essentially says that the dual of $\widehat{G}$ of a locally compact abelian group $G$ is canonically isomorphic to $G$) that the topological dual group of a discrete abelian group is compact. In particular, the dual group $\widehat{G/G_0}$ is compact. For more informations on locally compact abelian group, see for instance \cite{Morris}.

It is now clear that $G_0^\perp$ is a subset of the unimodular point spectrum $\sigma_p(\rho_{\mathcal{H}})\cap \widehat{G}$ of $\rho_{\mathcal{H}}$. It is straightforward to check that this group is isomorphic to $\widehat{G/G_0}$, and it is assumed in Theorem \ref{nonjamison} that the quotient group $G/G_0$ is infinite. A consequence is that the dual group $\widehat{G/G_0}$ is uncountable. For let $\mu$ denote the Haar measure of the compact group $K:=\widehat{G/G_0}$ and let us assume that this compact group is at most countable. Then 
$$
\sum_{k\in K}^{}\mu(\{k\})=\mu(K)<+\infty
$$ 
and by translation invariance of $\mu$, any two elements of $K$ have the same measure. Then $K$ is finite and the Pontryagin duality implies that $G/G_0\simeq \widehat{K}$ is also finite which contradicts our assumption. As a consequence, the set $\sigma_p(\rho_{\mathcal{H}})\cap \widehat{G}$ is uncountable. The proof of Theorem \ref{nonjamison} is now complete since $(\mathcal{H},\Vert\cdot\Vert)$ is a separable Hilbert space and $(\mathcal{H},\rho_{\mathcal{H}})$ is a representation which is partially bounded with respect to the sequence $(g_k)_{k\geqslant 0}$ with an uncountable unimodular point spectrum.  
\end{proof}

At the end of the proof of Theorem \ref{nonjamison}, we use the fact that the dual group of an infinite discrete group is automatically uncountable. There exists a stronger result in this way due to Kakutani \cite{Kakutani}: if $\mathcal{G}$ is an infinite discrete group then $|\widehat{\mathcal{G}}|=2^{|\mathcal{G}|}$ where $| C|$ stands for the cardinality of the set $C$.

The fact that our backward shift representation $\rho$ makes every element of $G$ into a bounded linear operator on $\ell^2(G,w)$ is a consequence of the \textit{sub-invariance} of the weighted function $w$ defined in \eqref{weight}.

\begin{Cor}\label{subinvariant}
Let $h$ be an element of $G$. There exists a positive real number $C_h$ such that 
$$
\forall g\in G,\qquad w(g+h)\leqslant C_h w(g)
$$
\end{Cor} 

\begin{proof}
The notations are the same as in the proof of Theorem \ref{nonjamison}. Let $k$ and $\ell$ be positive integers. We know from Lemma \ref{representation} that
$$
\forall f\in \ell^2(G,w),\qquad \Vert \rho(h_\ell)f\Vert\leqslant \sqrt{2}^{\ell+1}\Vert f\Vert
$$
Applying this inequality to the Dirac function $\delta_k$ of $\ell^2(G,w)$ defined by
$$
\forall g\in G,\qquad \delta_k(g)=\left\{\begin{array}{c l}
1 &\ \textrm{if }g=h_k\\
0 &\ \textrm{otherwise}
\end{array}\right.
$$
we get $\Vert \rho(h_\ell)\delta_k\Vert\leqslant \sqrt{2}^{\ell+1}\Vert \delta_k\Vert$ that is $w(h_k-h_\ell)\leqslant \sqrt{2}^{\ell+1}w(h_k)$ which proves Corollary \ref{subinvariant}.
\end{proof}

Since there is no $G$-Jamison sequences $(g_k)_{k\geqslant 0}$ such that the factor group $G/G_0$ is infinite according to Theorem \ref{nonjamison}, we now deal with the case of sequences of $G$ which generate \textit{big} subgroups $G_0$ of $G$. More precisely, we look for a characterization of $G$-Jamison sequences for which $G_0$ is a subgroup of $G$ of finite index in $G$.

\section{The theorem: characterization of Jamison sequences}

Let $G$ be a countably infinite discrete abelian group denoted by $G=\big\{h_k\,\big\vert\,k\geqslant 0\big\}$ and let $(g_k)_{k\geqslant 0}$ be a sequence of elements of $G$. As usual we denote by $G_0$ the subgroup of $G$ generated by $(g_k)_{k\geqslant 0}$ and we assume throughout this section that the quotient group $G/G_0$ is finite.

\begin{Prop}\label{wlog}
When $[G:G_0]<+\infty$, one can assume without loss of generality that $G=G_0$.
\end{Prop}

\begin{proof}
Let $n=[G:G_0]<+\infty$ and let $\{\ell_1,\dots,\ell_n\}$ be an enumeration of the group $G/G_0$. Then $G$ can be written as the disjoint union $G=\displaystyle \bigcup_{k=1}^{n}(\ell_k+G_0)$ and the new sequence $(\tilde{g}_k)_{k\geqslant 0}$ which is obtained by putting together $(g_k)_{k\geqslant 0}$ with the \textit{finite} sequence $(h_k)_{1\leqslant k\leqslant n}$ is such that $\tilde{G}_0=G$ where $\tilde{G}_0$ is the subgroup of $G$ generated by the sequence $(\tilde{g}_k)_{k\geqslant 0}$. Furthermore $(g_k)_{k\geqslant 0}$ is a $G$-Jamison sequence if and only if $(\tilde{g}_k)_{k\geqslant 0}$ has the same property.
\end{proof}

According to Proposition \ref{wlog} we can suppose in the sequel that $(g_k)_{k\geqslant 0}$ is a generate sequence of the group $G$. In particular, the map $d_{(g_k)}$ defined in \eqref{distance} is now a distance on $\widehat{G}$. The characterization of $G$-Jamison sequences runs as follows.

\begin{Theo}\label{characterization}
Let $(g_k)_{k\ge 0}$ be a sequence of elements of a countably infinite discrete abelian group $G$ denoted by $G=\big\{h_k\,\big\vert\,k\geqslant 0\big\}$. We assume that $(g_k)_{k\geqslant 0}$ is a generate sequence of the group $G$. Then the following assertions are equivalent$:$
\begin{enumerate}
\item[$(1)$] $(g_k)_{k\geqslant 0}$ is a $G$-Jamison sequence$;$
\item[$(2)$] for every uncountable subset $K$ of $\widehat{G}$, the metric space $(K,d_{(g_k)})$ is non-separable$;$
\item[$(3)$] for every uncountable subset $K$ of $\widehat{G}$, there exists $\varepsilon>0$ such that $K$ contains an uncountable $\varepsilon$-separated family for the distance $d_{(g_k)};$
\item[$(4)$] there exists $\varepsilon>0$ such that every uncountable subset $K$ of $\widehat{G}$ contains an uncountable $\varepsilon$-separated family for the distance $d_{(g_k)};$
\item[$(5)$] there exists $\varepsilon>0$ such that any two distinct characters of $G$ are $\varepsilon$-separated for the distance $d_{(g_k)}:$
$$
\forall (\chi,\varphi)\in \widehat{G}\times \widehat{G},\qquad \chi\ne \varphi\ \Longrightarrow\ \sup_{k\geqslant 0}|\chi(g_k)-\varphi(g_k)|\geqslant \varepsilon
$$
\end{enumerate}
\end{Theo}

\begin{proof}
The implication $(5) \Longrightarrow (1)$ is exactly Proposition \ref{conditionsuffisante} when $G=G_0$ since condition $(2)$ of Proposition \ref{conditionsuffisante} is clearly equivalent to condition $(5)$ of Theorem \ref{characterization} in this case. 

Furthermore we have the obvious implications $(5) \Longrightarrow (4) \Longrightarrow (3) \Longrightarrow (2)$ whereas $(2) \Longrightarrow (3)$ follows from the general theory of metric spaces. 

We now prove that $(3) \Longrightarrow (5)$. Suppose by contradiction that $(5)$ is not satisfied. Then one can find by induction a sequence $(\chi_n)_{n\geqslant 1}$ of elements of $\widehat{G}\setminus\{\mathbf{1}\}$ such that $d_{(g_k)}(\chi_1,\mathbf{1})<4^{-1}$ and
\begin{equation}
\forall n\geqslant 2,\qquad d_{(g_k)}(\chi_n,\mathbf{1})<4^{-n}d_{(g_k)}(\chi_{n-1},\overline{\chi_{n-1}})
\label{conditionchi}
\end{equation}
Since $\chi_{n-1}\in \widehat{G}\setminus\{\mathbf{1}\}$, there exists a non-negative integer $k_0$ such that $\chi_{n-1}(g_{k_0})\ne 1$. Then 
$$
|\chi_{n-1}(g_{k_0})-1|\leqslant d_{(g_k)}(\chi_{n-1},\mathbf{1})<\frac{1}{4}
$$
It follows that $\chi_{n-1}(g_{k_0})\ne -1$ and then $\chi_{n-1}\ne \overline{\chi_{n-1}}$. This implies that $d_{(g_k)}(\chi_{n-1},\overline{\chi_{n-1}})>0$ which proves the existence of $\chi_n\in \widehat{G}\setminus\{\mathbf{1}\}$ satisfying condition \eqref{conditionchi}. Moreover this condition implies that the sequence $(\chi_n)_{n\geqslant 1}$ is decreasing. We are now going to construct an uncountable subset $K$ of $\widehat{G}$ such that for every positive real number $\varepsilon$, every $\varepsilon$-separated family of $K$ is finite. For any finite sequence $(s_1,\dots,s_n)\in \{0,1\}^n$, we construct a character $\psi_{(s_1,\dots, s_n)}$ of $G$ in the following way. Let $\psi_{(0)}=\chi_1$ and $\psi_{(1)}=\overline{\chi_1}$. We have
$$
d_{(g_k)}(\psi_{(0)},\psi_{(1)})=d_{(g_k)}(\chi_1,\overline{\chi_1})>0
$$
Then we consider the characters 
$$
\psi_{(0,0)}=\psi_{(0)}\chi_2,\ \psi_{(0,1)}=\psi_{(0)}\overline{\chi_2}\qquad \textrm{and}\qquad \psi_{(1,0)}=\psi_{(1)}\chi_2,\ \psi_{(1,1)}=\psi_{(1)}\overline{\chi_2} 
$$
For $s_2\in \{0,1\}$, we have $d_{(g_k)}(\psi_{(0)},\psi_{(0,s_2)})=d_{(g_k)}(\chi_2,\mathbf{1})$ and then
$$
d_{(g_k)}(\psi_{(0)},\psi_{(0,s_2)})<\frac{1}{4^2}d_{(g_k)}(\chi_1,\overline{\chi_1})
$$
and in the same way
$$
d_{(g_k)}(\psi_{(1)},\psi_{(1,s_2)})<\frac{1}{4^2}d_{(g_k)}(\chi_1,\overline{\chi_1})
$$
Moreover 
$$
d_{(g_k)}(\psi_{(0,0)},\psi_{(0,1)})=d_{(g_k)}(\psi_{(1,0)},\psi_{(1,1)})=d_{(g_k)}(\chi_2,\overline{\chi_2})
$$
If $\psi_{(s_1,\dots,s_{n-1})}$ has already been defined then we set 
$$
\psi_{(s_1,\dots, s_{n-1},0)}=\psi_{(s_1,\dots, s_{n-1})}\chi_n\qquad \textrm{and}\qquad \psi_{(s_1,\dots, s_{n-1},1)}=\psi_{(s_1,\dots, s_{n-1})}\overline{\chi_n}
$$
Then 
\begin{equation}
d_{(g_k)}(\psi_{(s_1,\dots, s_{n-1})},\psi_{(s_1,\dots, s_n)})<\frac{1}{4^n}d_{(g_k)}(\chi_{n-1},\overline{\chi_{n-1}})
\label{limiteexiste}
\end{equation}
and 
$$
d_{(g_k)}(\psi_{(s_1,\dots,s_{n-1},0)},\psi_{(s_1,\dots, s_{n-1},1)})=d_{(g_k)}(\chi_n,\overline{\chi_n})
$$
We now define a character $\psi_s$ of $G$ for any infinite sequence $s=(s_n)_{n\geqslant 1}$ of zeroes and ones by putting
$$
\psi_s=\lim_{n\to +\infty}\psi_{s_1,\dots,s_n}
$$
The limit exists according to \eqref{limiteexiste} and for every positive integer $p$, we have the equality
\begin{equation}
\psi_s=\psi_{(s_1,\dots,s_p)}\prod_{j\geqslant p}\psi_{(s_1,\dots,s_{j+1})}\overline{\psi_{(s_1,\dots, s_j)}}
\label{caractereproduit}
\end{equation}
We now prove that the map $s \longmapsto \psi_s$ from $2^\omega$ to $\widehat{G}$ is one-to-one. Let $s=(s_1,\dots, s_{p-1},0,s_{p+1},\dots)$ and $s'=(s_1,\dots, s_{p-1},1,s_{p+1}',\dots)$ be two distinct elements of $2^\omega$. Using the representation \eqref{caractereproduit} of $\psi_s$ and $\psi_{s'}$, we get that $d_{(g_k)}(\psi_s,\psi_{s'})$ is equal to the supremum over the non-negative integers $k$ of
\begin{align*}
\Bigg| \psi_{(s_1,\dots,s_{p-1},0)}(g_k)\prod_{j\geqslant p}\psi_{(s_1,\dots,s_{j+1})}(g_k)&\overline{\psi_{(s_1,\dots, s_j)}(g_k)}\\
&-\psi_{(s_1,\dots,s_{p-1},1)}(g_k)\prod_{j\geqslant p}\psi_{(s_1',\dots,s_{j+1}')}(g_k)\overline{\psi_{(s_1',\dots, s_j')}(g_k)}\Bigg|
\end{align*}
Using the easy inequality $|\lambda_1\lambda_2-\mu_1\mu_2|\geqslant |\lambda_1-\mu_1|-|\lambda_2-\mu_2|$ which is valid for every complex numbers $\lambda_1,\mu_1,\lambda_2,\mu_2$ of modulus one, we find that this supremum is greater than
\begin{align*}
d_{(g_k)}(\psi_{(s_1,\dots, s_{p-1},0)}&,\psi_{(s_1,\dots, s_{p-1},1)})\\
&-\sup_{k\geqslant 0}\Bigg|\prod_{j\geqslant p}\psi_{(s_1,\dots, s_{j+1})}(g_k)\overline{\psi_{(s_1,\dots, s_j)}(g_k)}\overline{\psi_{(s_1',\dots, s_{j+1}')}(g_k)}\psi_{(s_1',\dots, s_j')}(g_k) -1\Bigg|
\end{align*}
Since a character takes it values in the unit circle, we have
\begin{align*}
\Bigg|\prod_{j\geqslant p}\psi_{(s_1,\dots, s_{j+1})}(g_k)&\overline{\psi_{(s_1,\dots, s_j)}(g_k)}\overline{\psi_{(s_1',\dots, s_{j+1}')}(g_k)}\psi_{(s_1',\dots, s_j')}(g_k) -1\Bigg|\\
&\leqslant \sum_{j\geqslant p}^{}d_{(g_k)}(\psi_{(s_1,\dots, s_{j+1})},\psi_{(s_1,\dots,s_j)})
+\sum_{j\geqslant p}^{}d_{(g_k)}(\psi_{(s_1',\dots, s_{j+1}')},\psi_{(s_1',\dots,s_j')})
\end{align*}
It follows from these inequalities that
\begin{align*}
d_{(g_k)}(\psi_s,\psi_{s'})&\geqslant d_{(g_k)}(\chi_p,\overline{\chi_{p}})-2\sum_{j\geqslant p}^{}4^{-j-1}d_{(g_k)}(\chi_j,\overline{\chi_j})\\
&\geqslant \Bigg(1-2\sum_{j\geqslant p}^{}4^{-j-1}\Bigg)d_{(g_k)}(\chi_p,\overline{\chi_{p}})
\end{align*}
since the sequence $\big(d_{(g_k)}(\chi_n,\overline{\chi_n})\big)_{n\geqslant 1}$ is decreasing. Hence 
$$
d_{(g_k)}(\psi_s,\psi_{s'})\geqslant \frac{5}{6}d_{(g_k)}(\chi_p,\overline{\chi_{p}})>0
$$
This proves that $K=\{\psi_s\,|\,s\in 2^\omega\}$ is an uncountable subset of $\widehat{G}$. Moreover, the same computation as above shows that
$$
d_{(g_k)}(\psi_s,\psi_{s'})\leqslant \frac{7}{6}d_{(g_k)}(\chi_p,\overline{\chi_p})
$$ 
for any positive integer $p$ and every elements $s$ and $s'$ of $2^\omega$ whose components coincide until the index $p$. Let $\varepsilon>0$. There exists a positive integer $p$ such that 
$$
d_{(g_k)}(\psi_s,\psi_{s'})\leqslant \frac{7}{6}d_{(g_k)}(\chi_p,\overline{\chi_p})\leqslant \varepsilon
$$ 
for arbitrary $s$ and $s'$ of $2^\omega$ such that $(s_1,\dots,s_p)=(s_1',\dots, s_p')$. We deduce from this that if $\psi_s$ and $\psi_{s'}$ are $\varepsilon$-separated for the distance $d_{(g_k)}$ then at least one the first $p$ coordinates of $s$ and $s'$ differs. If $\{\psi_{s_1},\dots, \psi_{s_m}\}$ is a family of characters of $G$ belonging to $K$ which are $\varepsilon$-separated for the distance $d_{(g_k)}$ then for every $(i,j)\in \{1,\dots, m\}^2$, we must have 
$$
(s_i(1),\dots, s_i(p))\ne(s_j(1),\dots,s_j(p))
$$
There is $2^p$ finite sequences of $\{0,1\}^p$, hence $m\leqslant 2^p$. Eventually every $\varepsilon$-separated of $K$ (for the distance $d_{(g_k)}$) is finite. This conclude the proof of $(3) \Longrightarrow (5)$.

It remains to prove that $(1)$ implies $(2)$. Let us assume that $(2)$ is not satisfied. Then there exists an uncountable subset $K$ of $\widehat{G}$ such that $(K,d_{(g_k)})$ is a separable metric space and we have to produce a representation $(X,\rho)$ of the group $G$ which is partially bounded with respect to the sequence $(g_k)_{k\geqslant 0}$ and such that the unimodular point spectrum $\sigma_p(\rho)\cap \widehat{G}$ of $\rho$ is uncountable. The beginning of the construction is the same as in the proof of Theorem \ref{nonjamison}: let us consider the representation $(\ell^2(G,w),\rho)$ where $\ell^2(G,w)$ is defined in \eqref{l2space} and \eqref{weight} and where $\rho$ is the translation representation. Our first task is to make our representation partially bounded with respect to the sequence $(g_k)_{k\geqslant 0}$. We use for this a renorming process of the space $\ell^2(G,w)$ which is inspired from the one used by Badea and Grivaux in the proof of Theorem 2.1 of \cite{BadeaGrivaux2}. The new norm will intrinsically depends on the sequence $(g_k)_{k\geqslant 0}$. 

\subsubsection*{Step 1: making $\rho$ partially bounded with respect to $(g_k)_{k\geqslant 0}$.}

We endow the space $\ell^2(G,w)$ with the new norm $\Vert\cdot \Vert_\ast$ defined by
\begin{equation}
\Vert f\Vert_*:=\max\Bigg(\Vert f\Vert,\sup_{j\geqslant 0} 2^{-j-1}\sup_{k_0,\dots,k_j\geqslant 0}\Bigg\Vert \prod_{i=0}^{j}\big(\rho(g_{k_i})-\mathrm{Id}\big)f\Bigg\Vert\Bigg)
\label{norm}
\end{equation}
for every element $f$ of $\ell^2(G,w)$ and we set $X:=\big\{f\in \ell^2(G,w)\,\big|\,\Vert f \Vert_\ast<+\infty\}$. Let us denote by $\rho_X$ the representation induced by $\rho$ on this space that is
$$
\rho_X : \left\{\begin{array}{c c c}
G & \longrightarrow & \mathcal{GL}(X)\\
g & \longmapsto & \rho(g)
\end{array}\right.
$$
The norm $\Vert\cdot\Vert_\ast$ makes the representation $\rho_X$ with the required boundedness property.

\begin{Lem}
The representation $(X,\rho_X)$ of the group $G$ is partially bounded with respect to the sequence $(g_k)_{k\geqslant 0}$. More precisely, 
$$
\sup_{k\geqslant 0}\Vert \rho(g_k)\Vert\leqslant 3
$$
\end{Lem}

\begin{proof}
For every element $f$ of $X$ and every non-negative integers $p$, $j$ and $k_0,\dots, k_j$, we have
\begin{align*}
2^{-j-1}\Bigg\Vert \prod_{i=0}^{j}\big(\rho(g_{k_i})-\mathrm{Id}\big)\rho(g_p)f\Bigg\Vert \leqslant 2\times 2^{-j-2}\Bigg\Vert \prod_{i=0}^{j}\big(&\rho(g_{k_i})-\mathrm{Id}\big)(\rho(g_p)-\mathrm{Id})f\Bigg\Vert\\
&+2^{-j-1}\Bigg\Vert \prod_{i=0}^{j}\big(\rho(g_{k_i})-\mathrm{Id}\big)f\Bigg\Vert
\end{align*}
and in the same way $\Vert\rho(g_p)(f)\Vert\leqslant 2\times 2^{-1}\Vert (\rho(g_p)-\mathrm{Id})(f)\Vert+ \Vert f\Vert$.
It follows from these inequalities that $\Vert \rho(g_p)(f)\Vert_\ast\leqslant 2\Vert f\Vert_\ast+\Vert f\Vert_\ast$ that is to say $\Vert \rho(g_p)f\Vert_\ast\leqslant 3\Vert f\Vert_\ast$ and the lemma is proved.
\end{proof}
We now have to make the space $(X,\Vert\cdot \Vert_\ast)$ separable and such that the unimodular point spectrum of the induced representation is uncountable.

\subsubsection*{Step 2: making the space $X$ separable and the unimodular point spectrum uncountable.}

The presence of the weights $2^{-j-1}$ in the norm $\Vert \cdot\Vert_\ast$ defined in \eqref{norm} ensures that every character $\chi$ of $G$ belongs to $X$. Moreover every element $\chi$ of $\widehat{G}$ is a unimodular eigenvalue of $\rho_X$ (with associated eigenvector $\rho$) since 
$$
\forall (g,h,\chi)\in G^2\times \widehat{G},\qquad \rho(g)\chi(h)=\chi(g+h)=\chi(g)\chi(h)
$$
We now use our assumption: since condition $(2)$ is not satisfied, there exists an uncountable subset $K$ of $\widehat{G}$ such that $(K,d_{(g_k)})$ is a separable metric space. In order to make the space $X$ separable, let us consider the new space 
$$
X_\ast:=\overline{\mathrm{span}}^{\Vert\cdot \Vert_\ast}\big\{\chi\,\big|\,\chi\in K\big\}
$$
which is equipped with the norm $\Vert\cdot\Vert_\ast$ and the representation $\rho_{X_\ast}$ induced by $\rho$ on the space $X_\ast$. The separability of $X_\ast$ follows from the above lemma.

\begin{Lem}\label{continuity}
The eigenvector field $E : \left\{\begin{array}{c c c}
K & \longrightarrow & X_\ast\\
\chi & \longmapsto & \chi
\end{array}\right.$ is continuous.
\end{Lem}

\begin{proof}
Let $\chi$ and $\varphi$ be two elements of $K$. The norm $\Vert \chi-\varphi\Vert_\ast$ is equal to the maximum between $\Vert \chi-\varphi\Vert$ and
$$
\sup_{j\geqslant 0}2^{-j-1}\sup_{k_0,\dots, k_j\geqslant 0}\Bigg\Vert\prod_{i=0}^{j}\big(\chi(g_{k_i})-1\big)\chi-\prod_{i=0}^{j}\big(\varphi(g_{k_i})-1\big)\varphi\Bigg\Vert
$$
For every non-negative integers $j$ and $k_0,\dots, k_j$, we have
\begin{align*}
\Bigg\Vert\prod_{i=0}^{j}\big(\chi(g_{k_i})-1\big)\chi-\prod_{i=0}^{j}\big(\varphi(g_{k_i})-1\big)\varphi\Bigg\Vert
&\leqslant d_j(\chi,\mathbf{1})\Vert \chi- \varphi\Vert+ d_j(\chi,\varphi)\Vert \varphi\Vert
\end{align*}
where 
$$
d_j(\chi,\varphi):=\sup_{k_0,\dots, k_j\geqslant 0} \Bigg\vert\prod_{i=0}^{j}\big(\chi(g_{k_i})-1\big)-\prod_{i=0}^{j}\big(\varphi(g_{k_i})-1\big)\Bigg\vert
$$
It is rather easy to prove that 
\begin{equation}\label{dj}
d_j(\chi,\varphi)\leqslant 2^jd_{(g_k)}(\chi,\varphi)+2d_{j-1}(\chi,\varphi)
\end{equation}
This is an easy consequence of the following identity which is valid for every positive integer $j$ and every non-negative integers $k_0,\dots, k_j$:
\begin{align*}
\prod_{i=0}^{j}&\big(\chi(g_{k_i})-1\big)-\prod_{i=0}^{j}\big(\varphi(g_{k_i})-1\big)\\
&=\big(\chi(g_{k_0})-
\varphi(g_{k_0})\big)\prod_{i=1}^{j}\big(\chi(g_{k_i})-1\big)+\big(\varphi(g_{k_0})-1\big)\Bigg(\prod_{i=1}^{j}\big(\chi(g_{k_i})-1\big)&-\prod_{i=1}^{j}\big(\varphi(g_{k_i})-1\big)\Bigg)
\end{align*}
By an induction argument, it follows from \eqref{dj} that 
$$
d_j(\chi,\varphi)\leqslant (j+1)2^jd_{(g_k)}(\chi,\varphi)
$$
for every non-negative integer $j$. We then find that there exists a positive constant $C$ such that
\begin{equation}\label{continuityE1}
\Vert \chi-\varphi\Vert_\ast\leqslant C\big(\Vert \chi-\varphi\Vert+d_{(g_k)}(\chi,\varphi)\big)
\end{equation}
Recall that the group $G$ is enumerated as $G=\{h_n\,|\,n\geqslant 0\}$. Let us fix a positive real number $\varepsilon$. Then there exists a positive integer $N_\varepsilon$ such that
\begin{align}\label{continuityE2}
\notag\Vert \chi-\varphi\Vert^2 =\sum_{g\in G}^{}\vert \chi(g)-\varphi(g)\vert^2 w(g)&=\sum_{n=1}^{+\infty} \vert \chi(h_n)-\varphi(h_n)\vert^2 w(h_n)\\
&\leqslant \sum_{n=1}^{N_\varepsilon} \vert \chi(h_n)-\varphi(h_n)\vert^2 w(h_n)+\varepsilon
\end{align}
Let $n$ be a positive integer less or equal than $N_\varepsilon$. Since $w$ is a probability measure on $G$ we have $w(h_n)\leqslant 1$. Moreover we know that $(g_k)_{k\geqslant 0}$ is a generate sequence of the group $G$. Hence one can find integers $\ell_n\geqslant 0$ and $\big(\alpha_0^{(n)},\dots, \alpha_{\ell_n}^{(n)}\big)\in \mathbb{Z}^{\ell_n+1}$ such that 
$$
h_n=\sum_{0\leqslant k\leqslant \ell_n}^{}\alpha_k^{(n)}g_k
$$ 
and then
\begin{align}\label{continuityE3}
\notag\vert \chi(h_n)-\varphi(h_n)\vert &\leqslant \Bigg\vert \prod_{k=0}^{\ell_n}\chi(g_k)^{\alpha_k^{(n)}}-\prod_{k=0}^{\ell_n}\varphi(g_k)^{\alpha_k^{(n)}} \Bigg\vert\\
\notag &\leqslant \sum_{k=0}^{\ell_n}\big\vert \alpha_k^{(n)}\big\vert\, \big\vert\chi(g_k)-\varphi(g_k)\big\vert\\
&\leqslant \Bigg(\sum_{k=0}^{\ell_n}\big\vert \alpha_k^{(n)}\big\vert\Bigg)d_{(g_k)}(\chi,\varphi)
\end{align}
It then follows from \eqref{continuityE1}, \eqref{continuityE2} and \eqref{continuityE3} that for every $\varepsilon >0$, there exists a positive constant $C_\varepsilon$ such that 
$$
\Vert \chi-\varphi\Vert_\ast\leqslant C_\varepsilon d_{(g_k)}(\chi,\varphi)+\varepsilon
$$
and Lemma \ref{continuity} is proved.
\end{proof}
According to Lemma \ref{continuity} and the separability of $(K,d_{(g_k)})$, the space $X_\ast$ is separable and the representation $(X_\ast,\rho_{X_\ast})$ remains partially bounded with respect to the sequence $(g_k)_{k\geqslant 0}$ and has an uncountable unimodular point spectrum (since it contains the uncountable set $K$). This prove that $(1)$ implies $(2)$. The proof of Theorem \ref{characterization} is now complete.
\end{proof}

\nocite{*}
\bibliography{biblio1}
\bibliographystyle{plain}

\noindent \textit{Vincent Devinck}\newline
\noindent \textit{E-mail address}: \texttt{vincedevinck@gmail.com}

\end{document}